\documentclass[12pt,a4paper]{article}
\usepackage{t1enc,amsthm,amsmath,amssymb,url,multirow,rotating}
\usepackage[mathcal,mathscr]{eucal}
\usepackage[a4paper,left=2.6cm,right=2.6cm,top=3.1cm,bottom=3.0cm]{geometry}
\usepackage{fancyhdr}
\usepackage{calc}
\usepackage{tikz}
\usepackage[english]{babel}
\usepackage{breqn}
\usepackage{titlesec}
\usepackage{hyperref,enumitem,thmtools,etoolbox}

\usepackage[numbers]{natbib}

\usetikzlibrary{backgrounds}
\usetikzlibrary{shapes,decorations,decorations.pathreplacing,decorations.pathmorphing}
\usetikzlibrary{decorations.markings}
\tikzstyle{vertex}=[circle, draw, minimum size=1pt]
\newcommand{\vertex}{\node[vertex]}

\author{B\'alint V\'as\'arhelyi\footnote{Szegedi Tudom\'anyegyetem, TTIK, Szeged, 6720, Hungary. E-mail: mesti@math.u-szeged.hu}}
\title{An Estimation of the Size of Non-Compact Suffix Trees}
\date{}

\newtheorem{theorem}{Theorem}
\newtheorem{definition}[theorem]{Definition}
\newtheorem{lemma}[theorem]{Lemma}
\newtheorem{corollary}[theorem]{Corollary}
\newtheorem{observation}[theorem]{Observation}
\theoremstyle{definition}
\newtheorem{algorithm}[theorem]{Algorithm}

\AtEndEnvironment{definition}{\mbox{{ }}\null\hfill$\diamond$}

\newcommand{\summ}{\sum\limits}			

\newcommand{\ee}{\mathbb{E}}
\newcommand{\pp}{\mathcal{P}}
\renewcommand{\ll}{\mathcal{L}}

\newcommand{\ttt}{\mathcal{T}}
\renewcommand{\(}{\left(}
\renewcommand{\)}{\right)}

\begin{document}

\maketitle
\normalsize


\begin{abstract}
A suffix tree is a data structure used mainly for pattern matching. It is known that the space complexity of simple suffix trees is quadratic in the length of the string. By a slight modification of the simple suffix trees one gets the compact suffix trees, which have linear space complexity. The motivation of this paper is the question whether the space complexity of simple suffix trees is quadratic not only in the worst case, but also in expectation.
\end{abstract}

\section{Introduction}

A suffix tree is a powerful data structure which is used for a large number of combinatorial problems involving strings. Suffix tree is a structure for compact storage of the suffixes of a given string. The compact suffix tree is a modified version of the suffix tree, and it can be stored in linear space of the length of the string, while the non-compact suffix tree is quadratic (see \cite{gusfield,mccreight,ukkonen,weiner}).

The notion of suffix trees was first introduced by Weiner \citep{weiner}, though he used the name compacted bi-tree. Grossi and Italiano mention that in the scientific literature, suffix trees have been rediscovered many times, sometimes under different names, like compacted bi-tree, prefix tree, PAT tree, position tree, repetition finder, subword tree etc. \cite{grossi} .

Linear time and space algorithms for creating the compact suffix tree were given soon by Weiner \cite{weiner}, McCreight \cite{mccreight}, Ukkonen \cite{ukkonen}, Chen and Sciferas \cite{sciferas}
 and others.
 
The statistical behaviour of suffix trees has been also studied. Most of the studies consider improved versions.

The average size of compact suffix trees was examined by Blumer, Ehrenfeucht and Haussler \cite{blumer}. They proved that the average number of nodes in the compact suffix tree is asymptotically the sum of an oscillating function and a small linear function.

An important question is the height of suffix trees, which was answered by Devroye, Szpankowski and Rais \cite{devroye}, who proved that the expected height is logarithmic in the length of the string.

The application of suffix trees is very wide. We mention but only a few examples. Apostolico et al. \cite{apostolico} mention that these structures are used in text searching, indexing, statistics, compression. In computational biology, several algorithms are based on suffix trees. Just to refer a few of them, we mention the works of H\"ohl et al. \cite{hohl}, Adebiyi et al. \cite{adebiyi} and Kaderali et al. \cite{kaderali}

Suffix trees are also used for detecting plagiarism \cite{apostolico}, in cryptography \cite{oconnor,rodeh}, in data compression \cite{fiala,fraser,rodeh} or in pattern recognition \cite{tanimoto}.

For the interested readers further details on suffix trees, their history and their applications can be found in \cite{apostolico}, in \cite{grossi} and  in \cite{gusfield}, which sources we also used for the overview of the history of suffix trees.

It is well-known that the non-compact suffix tree can be quadratic in space as we referred before. In our paper we are setting a lower bound on the average size, which is also quadratic.

\section{Preliminaries}

Before we turn to our results, let us define a few necessary notions.

\begin{definition}
An \emph{alphabet} $\Sigma$ is a set of different characters. The \emph{size} of an alphabet is the size of this set, which we denote by $\sigma(\Sigma)$, or more simply $\sigma$. A string $S$ is over the alphabet $\Sigma$ if each character of $S$ is in $\Sigma$.
\end{definition}

\begin{definition}
Let $S$ be a string. $S[i]$ is its \emph{i}th character, while $S[i,j]$ is a \emph{substring} of $S$, from $S[i]$ to $S[j]$, if $j\geq i$, else $S[i,j]$ is the empty string. Usually $n(S)$ (or $n$ if there is no danger of confusion) denotes the \emph{length} of the string.
\end{definition}

\begin{definition}
The \emph{suffix tree} of $S$ is a rooted directed tree with $n$ leaves, where $n$ is the length of $S$.

Its structure is the following:

\setlength{\leftskip}{3ex} Each edge $e$ has a \emph{label} $\ell(e)$, and the edges from a node $v$ have different labels (thus, the suffix tree of a string is unique). If we concatenate the edge labels along a path $\pp$, we get the \emph{path label} $\ll(\pp)$.

We denote the path from the root to the leaf $j$ by $\pp(j)$. The edge labels are such that $\ll(j) = \ll(\pp(j))$ is $S[j,n]$ and a $\$$ sign at the end. The definition becomes more clear if we check the example on \autoref{fig1} and \autoref{naiv_algorithm}.

\end{definition}

\setlength{\leftskip}{0ex}

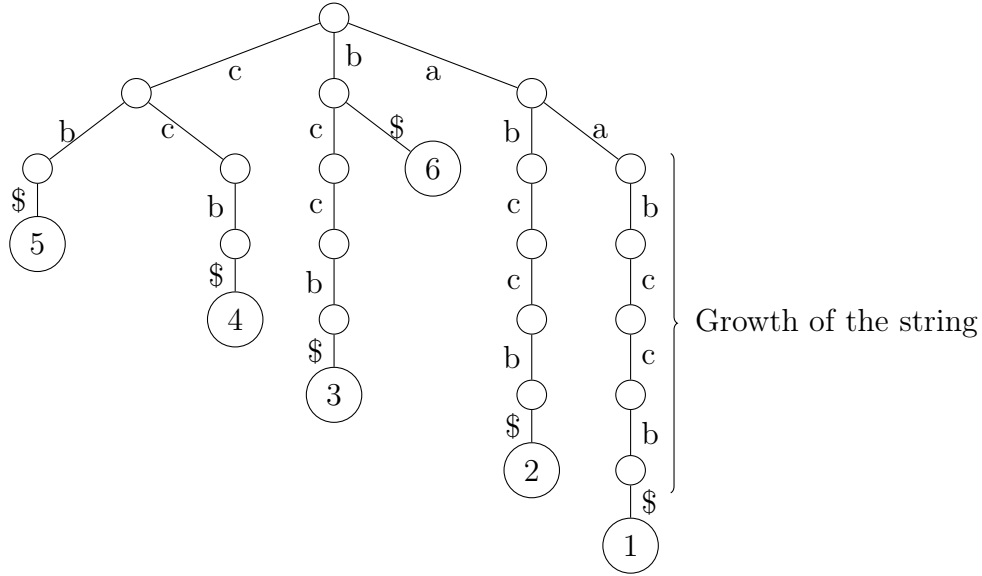
\begin{figure}[ht]
	\centering
\begin{tikzpicture}[x=1.3cm, y=1cm, every edge/.style={ draw, }]
	\vertex (r) at (3,3) {};
	\vertex (i1) at (1,2) {};
	\vertex (i11) at (0,1) {};
	\vertex (i12) at (2,1) {};
	\vertex (i121) at (2,0) {};
	\vertex (i2) at (3,2) {};
	\vertex (i21) at (3,1) {};
	\vertex (i211) at (3,0) {};
	\vertex (i2111) at (3,-1) {};
	\vertex (i3) at (5,2) {};
	\vertex (i31) at (5,1) {};
	\vertex (i311) at (5,0) {};
	\vertex (i3111) at (5,-1) {};
	\vertex (i31111) at (5,-2) {};
	\vertex (i32) at (6,1) {};
	\vertex (i321) at (6,0) {};
	\vertex (i3211) at (6,-1) {};
	\vertex (i32111) at (6,-2) {};
	\vertex (i321111) at (6,-3) {};
	\vertex (5) at (0,0) {$5$};
	\vertex (4) at (2,-1) {$4$};
	\vertex (3) at (3,-2) {$3$};
	\vertex (6) at (4,1) {$6$};
	\vertex (2) at (5,-3) {$2$};
	\vertex (1) at (6,-4) {$1$};
	\path
		(r) edge node[below]{c} (i1)
		(r) edge node[right]{b} (i2)
		(r) edge node[below]{a} (i3)
		(i1) edge node[left]{b} (i11)
		(i11) edge node[left]{\$} (5)
		(i1) edge node[left]{c} (i12)
		(i12) edge node[left]{b} (i121)(4)
		(i121) edge node[left]{\$} (4)
		(i2) edge node[left]{c} (i21)
		(i21) edge node[left]{c} (i211)
		(i211) edge node[left]{b} (i2111)
		(i2111) edge node[left]{\$} (3)
		(i2) edge node[right]{\$} (6)
		(i3) edge node[left]{b} (i31)
		(i31) edge node[left]{c} (i311)
		(i311) edge node[left]{c} (i3111)
		(i3111) edge node[left]{b} (i31111)
		(i31111) edge node[left]{\$} (2)
		(i3) edge node[right]{a} (i32)
		(i32) edge node[right]{b} (i321)
		(i321) edge node[right]{c} (i3211)
		(i3211) edge node[right]{c} (i32111)
		(i32111) edge node[right]{b} (i321111)
		(i321111) edge node[right]{\$} (1)
	;
	\draw[decoration={brace,raise=15pt},decorate]
	(6,1.2) -- node[right=20pt] {Growth of the string} (6,-3.3);
\end{tikzpicture}
\caption{\protect\label{fig1} Suffix tree of string $aabccb$}
\end{figure}

A naive algorithm for constructing the suffix tree is the following:

\begin{algorithm}
	\label{naiv_algorithm}
	Let $S$ be a string of length $n$. Let $j=1$ and $T$ be a tree of one vertex $r$ (the root of the suffix tree).
	\setlength{\parindent}{24pt}
	\begin{enumerate}[label=\textbf{Step \arabic*:}]
		\setlength{\itemindent}{24pt}
		\item Consider $X = S[j,n]+\$$. Set $i=0$, and $v=r$.
		\item If there is an edge $vu$ labelled $X[i+1]$, then set $v=u$ and $i=i+1$.
		\item Repeat Step 2 while it is possible.
		\item If there is no such an edge, add a path of $n-j-i+2$ edges from $v$, with labels corresponding to $S[j+i,n]+\$$, consecutively on the edges. At the end of the path, number the leaf with $j$.
		\item Set $j=j+1$, and if $j\leq n$, go to Step 1.\null\hfill$\diamond$
	\end{enumerate}
	\setlength{\parindent}{0pt}
\end{algorithm}

Notice that in \autoref{naiv_algorithm} a leaf always remain a leaf, as \$ (the last edge label before a leaf) is not a character in $S$.

\begin{definition}
The \emph{compact suffix tree} is a modified version of the suffix tree. We get it from the suffix tree by compressing its long branches.

\end{definition}
The structure of the compact suffix tree is basically similar to that of the suffix tree, but an edge label can be longer than one character, and each internal node (i.e. not leaf) must have at least two children. For an example see \autoref{fig2}.

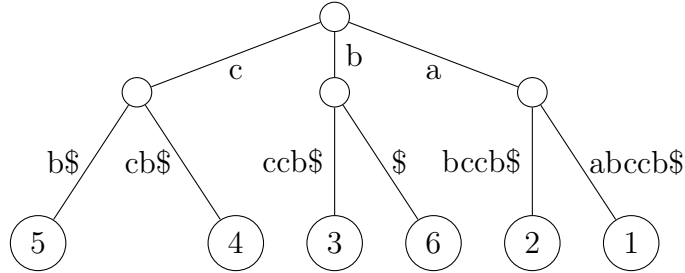
\begin{figure}[ht]
	\centering
	\begin{tikzpicture}[x=1.3cm, y=1cm, every edge/.style={ draw, }]
	\vertex (r) at (3,3) {};
	\vertex (i1) at (1,2) {};
	\vertex (i2) at (3,2) {};
	\vertex (i3) at (5,2) {};
	\vertex (5) at (0,0) {$5$};
	\vertex (4) at (2,0) {$4$};
	\vertex (3) at (3,0) {$3$};
	\vertex (6) at (4,0) {$6$};
	\vertex (2) at (5,0) {$2$};
	\vertex (1) at (6,0) {$1$};
	\path
	(r) edge node[below]{c} (i1)
	(r) edge node[right]{b} (i2)
	(r) edge node[below]{a} (i3)
	(i1) edge node[left]{b\$} (5)
	(i1) edge node[left]{cb\$} (4)
	(i2) edge node[left]{ccb\$} (3)
	(i2) edge node[right]{\$} (6)
	(i3) edge node[left]{bccb\$} (2)
	(i3) edge node[right]{abccb\$} (1)
	;
	\end{tikzpicture}
	\caption{\protect\label{fig2} Compact tree of string $aabccb$}
\end{figure}

With a regard to suffix trees, we can define further notions for strings.

\begin{definition}
Let $S$ be a string, and $\ttt$ be its (non-compact) suffix tree.

A \emph{natural direction} of $\ttt$ is that all edges are directed from the root towards the leaves. If there is a directed path from $u$ to $v$, then $v$ is a \emph{descendant} of $u$ and $u$ is an \emph{ancestor} of $v$.

We say that \emph{the growth of} $S$ (denoted by $\gamma(S)$) is one less than the shortest distance of leaf 1 from an internal node $v$ which has at least two children (including leaf 1), that is, we count the internal nodes on the path different from $v$. If leaf $j$ is a descendant of $v$, then the common prefix of $S[j,n]$ and $S[1,n]$ is the longest among all $j$'s.

\end{definition}
If we consider the string $S=aabccb$, the growth of $S$ is 5, as it can be seen on \autoref{fig1}.


An important notion is the following one.

\begin{definition}
Let $\Omega(n,k,\sigma)$ be the number of strings of length $n$ with growth $k$ over an alphabet of size $\sigma$.
\end{definition}

Observe that the connection between the growth and the number of nodes in a suffix tree is the following:
\begin{observation}
\label{observation8}
If we construct the suffix tree of $S$ by using \autoref{naiv_algorithm}, we get that the sum of the growths of $S[n-1,n],S[n-2,n],\ldots,S[1,n]$ is a lower bound to the number of nodes in the final suffix tree. In fact, there are only two more internal nodes, the root vertex, the only node on the path to leaf $n$, and we have the leaves.
\end{observation}

In the proofs we will need the notion of period and of aperiodic strings.

\begin{definition}
Let $S$ be a string of length $n$. We say that $S$ is \emph{periodic} with period $d$, if there is a $d|n$ for which $S[i] = S[i+d]$ for all $i\leq n-d$. Otherwise, $S$ is \emph{aperiodic}.

The \emph{minimal period} of $S$ is the smallest $d$ with the property above.
\end{definition}

\begin{definition}
$\mu(j,\sigma)$ is the number of $j$-length aperiodic strings over an alphabet of size $\sigma$.
\end{definition}

A few examples for the number of aperiodic strings are given in \autoref{tab:mux}.

\begin{table*}[hbt]
	\centering
	\begin{tabular}{|l|cc|cc|cc|cc|}\hline
		$\sigma$&$\mu(1,\sigma)$&$\mu(2,\sigma)$&$\mu(3,\sigma)$&$\mu(4,\sigma)$&$\mu(5,\sigma)$&$\mu(6,\sigma)$&$\mu(7,\sigma)$&$\mu(8,\sigma)$\\\hline
		2&2& 6& 12& 30& 54& 126& 240& 504\\\hline
		3&3&6&24&72&240&696&2184&648\\\hline
		4&4&12&60&240&1020&4020&16380&65280\\\hline
		5&5&20&120&600&3120&15480&78120&390000\\\hline
	\end{tabular}
	\caption{Number of aperiodic strings for small alphabets. $\sigma$ is the size of the alphabet, and $\mu(j,\sigma)$ is the number of aperiodic strings of length $j$}
	\label{tab:mux}
\end{table*}

\section{Main results}

Our main results are formulated in the following theorems.

\begin{theorem}
	\label{tetel}
	On an alphabet of size $\sigma$ for all $n\geq 2k$, $\Omega(n,k,\sigma) \leq \phi(k,\sigma)$ for some function $\phi$.
\end{theorem}

\begin{theorem}
\label{tetel:2}
There is a $c>0$ and an $n_0$ such that for any $n>n_0$ the following is true. Let $S'$ be a string of length $n-1$, and $S$ be a string obtained from $S'$ by adding a character to its beginning chosen uniformly random from the alphabet. Then the expected growth of $S$ is at least $c\cdot n$.
\end{theorem}

\begin{theorem}
\label{tetel:3}
There is a $d>0$ that for any $n>n_0$ (where $n_0$ is the same as in \autoref{tetel:2}) the following holds. On an alphabet of size $\sigma$ the simple suffix tree of a random string $S$ of length $n$ has at least $d\cdot n^2$ nodes in expectation.
\end{theorem}

\section{Proofs}

\begin{proof} (\autoref{tetel:3})
	
Considering \autoref{observation8} we have that the expected size of the simple suffix tree of a random string $S$ is at least
\begin{equation}
\ee\summ_{m=1}^n(\gamma(S[n-m,n])) \geq \summ_{m=1}^n\ee(\gamma(S[n-m,n])).
\end{equation}

If $m\leq n_0$, \autoref{tetel:2} is obvious. If $m>n_0$, we can divide the sum into two parts:

\begin{equation}\summ_{m=1}^n\ee(\gamma(S[n-m,n]))  = \summ_{m=1}^{n_0}\ee(\gamma(S[n-m,n])) + \summ_{m=n_0+1}^{n}\ee(\gamma(S[n-m,n])).
\end{equation}

The first part of the sum is a constant, while the second part can be estimated with \autoref{tetel:2}:
\begin{equation}
\summ_{m=n_0+1}^{n}\ee(\gamma(S[n-m,n])) \geq \summ_{m=n_0+1}^{n} cn = d\cdot n^2.\end{equation}

This proves \autoref{tetel:3}.

\end{proof}

First, we show a few lemmas about the number of aperiodic strings. \autoref{lemma:1} can be found in \cite{gilbert} or in \cite{cook}, but we give a short proof also here.

\begin{lemma}
	For all $j>0$ integer and for all alphabet of size $\sigma$ the number of aperiodic strings is
	\label{lemma:1}
	\begin{equation}
	\mu(j,\sigma) = \sigma^j - \summ_{\substack{d|j\\d\neq j}}\mu(d,\sigma).
	\end{equation}
\end{lemma}

\begin{proof}

	$\mu(1,\sigma) = \sigma$ is trivial.
	
	There are $\sigma^j$ strings of length $j$. Suppose that a string is periodic with minimal period $d$. This implies that its first $d$ characters form an aperiodic string of length $d$, and there are $\mu(d,\sigma)$ such strings. This finishes the proof.
\end{proof}

Specially, if $p$ is prime, then $\mu(p,\sigma) = \sigma^p-\sigma$.

\begin{corollary}
	\label{lemma:3}
	If $p$ is prime and $t\in\mathbb{N}$, then \begin{math}\mu\(p^t,\sigma\) = \sigma^{p^t}-\sigma^{p^{t-1}}\end{math} for all alphabet of size $\sigma$.
\end{corollary}
\begin{proof}
	We count the aperiodic strings of length $p^t$. There are $\sigma^{p^t}$ strings. Consider the minimal period of the string, i.e. the period which is aperiodic. If we exclude all minimal periods of length $k$, we exclude $\mu(k,\sigma)$ strings. This yields the following equality:
	\begin{equation}
	\mu\(p^t,\sigma\) = \sigma^{p^t} - \summ_{1\leq s<t} \mu\(p^s,\sigma\).
	\label{eq2}
	\end{equation}
	With a few transformations and using \autoref{lemma:1}, we have that \eqref{eq2} is equal to
	\begin{equation}
\sigma^{p^t} - \mu\(p^{t-1},\sigma\) - \summ_{1\leq s < t-1} \mu\(p^s,\sigma\) =\sigma^{p^t} - \sigma^{p^{t-1}} + \summ_{1\leq s < t-1} \mu\(p^s,\sigma\) - \summ_{1\leq s < t-1} \mu\(p^s,\sigma\),
	\end{equation}
	which is
	\begin{equation}
\sigma^{p^t}-\sigma^{p^{t-1}}.
	\end{equation}

\end{proof}
\begin{lemma}
For all $j>1$ and for all alphabet of size $\sigma$ , $\mu(j,\sigma)\leq \sigma^j-\sigma$.
\label{lemma:5}
\end{lemma}

\begin{proof}
From \autoref{lemma:1} we have $\mu(j,\sigma) = \sigma^j-\summ_{\substack{d|j\\d\neq j}}\mu(d,\sigma)$. 
Considering $\mu(d,\sigma) \geq 0$ and $\mu(1,\sigma) = \sigma$, we get the claim of the lemma.
\end{proof}

\begin{lemma}
\label{lemma:6}
For all $j\geq 1$, and for all  alphabet of size $\sigma$
\begin{equation}
\mu(j,\sigma) \geq\sigma(\sigma-1)^{j-1}.
\end{equation}
\end{lemma}

\begin{proof}

We prove by induction. For $j = 1$ the claim is obvious, as $\mu(1,\sigma) = \sigma$.

Suppose we know the claim for $j-1$. Consider $\sigma(\sigma-1)^{j-2}$ aperiodic strings of length $j-1$. Now, for any of these strings there is at most one character by appending that to the end of the string we receive a periodic string of length $j$. Therefore we can append at least $\sigma-1$ characters to get an aperiodic string, which gives the desired result.

\end{proof}

\begin{observation}
	Observe that if the growth of $S$ is $k$, then there is a $j$ such that $S[1,n-k] = S[j+1,j+n-k]$. 
	For example, if the string is $abcdefabcdab$ ($n=12$), one can check that the growth is 8 (the new branch in the suffix tree which ends in leaf $1$ starts after $abcd$), and with $j=6$ we have $S[1,4] = S[7,10] = abcd$.
	
	The reverse of this observation is that if there is a $j<n$ such that $S[1,n-k] = S[j+1,j+n-k]$, then the growth is \textit{at most} $k$, as $S[j+1,n]$ and $S[1,n]$ shares a common prefix of length $n-k$, thus, the paths to the leaves $j+1$ and $n$ share $n-k$ internal nodes, and at most $k$ new internal nodes are created.
	\label{obs_final}
\end{observation}

\begin{proof}(\autoref{tetel})
For proving the theorem we count the number of strings with growth $k$ for $n\geq 2k$.

First, we fix $j$, and then count the number of possible strings where the growth occurs such that $S[1,n-k] = S[j+1,j+n-k]$ for that fixed $j$. Note that by this way, we only have an upper bound for this number, as we might found an $\ell$ such that $S[1,n-k+1] = S[\ell+1,\ell+n-k+1]$.

We know that $j\leq k$, otherwise $S[j+1, j+n-k]$ does not exist.

If $j=k$, then we know $S[1,n-k] = S[k+1,n]$.

$S[1,k]$ must be aperiodic. Suppose the opposite and let $S[1,k] = p\ldots p$, where $p$ is the minimal period, and its length is $d$. Then $S[k+1,n] = p\ldots p$. Obviously, in this case $S[1,n-d] = S[d+1,n]$, which by \autoref{obs_final} means that the growth would be at most $d$. See also \autoref{thm11:j=k}.

Therefore this case gives us at most $\mu(k)$ strings of growth $k$.

\begin{figure}[h!]
	\centering
	\begin{tikzpicture}
	\draw (-0.1,0) -- (-0.1,1.4) node[right,at start] {1};
	\draw (0,1) rectangle (1,0.5) node[pos=.5] {$p$} ;
	\draw (1.1,1) rectangle (2.1,0.5) node[pos=.5] {$p$};
	\draw[fill] (2.4,0.75) circle [radius=0.025];
	\draw[fill] (2.6,0.75) circle [radius=0.025];
	\draw[fill] (2.8,0.75) circle [radius=0.025];
	\draw (3.1,1) rectangle (4.1,0.5) node[pos=.5] {$p$};
	\draw (4.2,0) -- (4.2,1.4) node[left,at start] {$k$};
	\draw (4.3,1) rectangle (5.3,0.5) node[pos=.5] {$p$} ;
	\draw[fill] (5.6,0.75) circle [radius=0.025];
	\draw[fill] (5.8,0.75) circle [radius=0.025];
	\draw[fill] (6.0,0.75) circle [radius=0.025];
	\draw (6.3,1) rectangle (7.3,0.5) node[pos=.5] {$p$};
	\draw (7.4,1) rectangle (8.4,0.5) node[pos=.5] {$p$};
	\draw (8.5,0) -- (8.5,1.4) node[left,at start] {$n$};
	
	\draw[blue,dashed, thick] (-0.05,0.25) rectangle (7.35,1.15);
	\draw[violet,dashed, thick] (1.05,0.4) rectangle (8.45,1.35);	
	\end{tikzpicture}
	\caption{Proof of \autoref{tetel}, case $j=k$}
	\protect\label{thm11:j=k}
\end{figure}
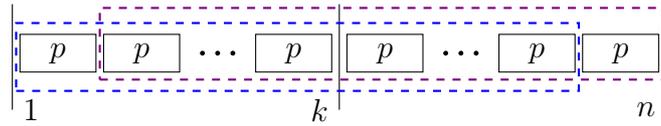

If $j<k$, then we have $S[1,n-k] = S[j+1,j+n-k]$.

First, we note that $S[1,j]$ must be aperiodic. Suppose the opposite and let $S[1,j] = p\ldots p$, where $p$ is the minimal period, and its length is $d$. Then 
\begin{equation}
S[j+1,2j] = S[2j+1,3j] = \ldots = p\ldots p,
\end{equation}
which means that \begin{equation}
S\left[1,\left\lfloor\frac{k}{j}\right\rfloor\cdot j\right] = S\left[j+1,j+\left\lfloor\frac{k}{j}\right\rfloor\cdot j\right] = p\ldots p.
\end{equation}

This implies that $S[1,j+n-k] = p\ldots pp'$, where $p'$ is a prefix of $p$. However, $S[1,j+n-k-d] = S[d,j+n-k]$ is true, and using \autoref{obs_final}, we have that $\gamma(S) \leq n - (j+n-k) + d = k-j+d<k$, which is a contradiction.

Further, $S[j+n-k+1]$ must not be the same as $S[k+1]$, which means that this character can be chosen $\sigma-1$ ways.

Therefore this case gives us at most $\mu(j)(\sigma-1)\sigma^{k-j-1}$ strings of growth $k$ for each $j$. 

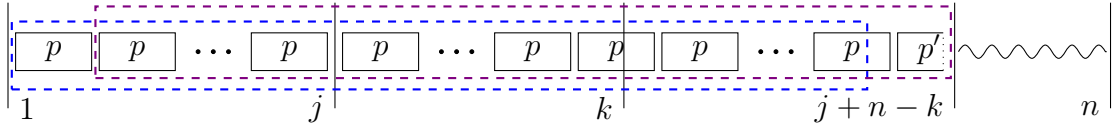
\begin{figure}[h!]
	\centering
	\begin{tikzpicture}
	\draw (-0.1,0) -- (-0.1,1.4) node[right,at start] {1};
	\draw (0,1) rectangle (1,0.5) node[pos=.5] {$p$} ;
	\draw (1.1,1) rectangle (2.1,0.5) node[pos=.5] {$p$};
	\draw[fill] (2.4,0.75) circle [radius=0.025];
	\draw[fill] (2.6,0.75) circle [radius=0.025];
	\draw[fill] (2.8,0.75) circle [radius=0.025];
	\draw (3.1,1) rectangle (4.1,0.5) node[pos=.5] {$p$};
	\draw (4.2,0) -- (4.2,1.4) node[left,at start] {$j$};
	\draw (4.3,1) rectangle (5.3,0.5) node[pos=.5] {$p$} ;
	\draw[fill] (5.6,0.75) circle [radius=0.025];
	\draw[fill] (5.8,0.75) circle [radius=0.025];
	\draw[fill] (6.0,0.75) circle [radius=0.025];
	\draw (6.3,1) rectangle (7.3,0.5) node[pos=.5] {$p$};
	\draw (7.4,1) rectangle (8.4,0.5) node[pos=.5] {$p$};
	\draw (8,0) -- (8,1.4) node[left,at start] {$k$};
	\draw (8.5,1) rectangle (9.5,0.5) node[pos=.5] {$p$};
	\draw[fill] (9.8,0.75) circle [radius=0.025];
	\draw[fill] (10,0.75) circle [radius=0.025];
	\draw[fill] (10.2,0.75) circle [radius=0.025];	
	\draw (10.5,1) rectangle (11.5,0.5) node[pos=.5] {$p$};
	\draw (12.2,0.5) -- (11.6,0.5) -- (11.6,1) -- (12.2,1) node[yshift=-6.25,xshift=-5.2] {$p'$};
	\draw[dotted] (12.2,0.5) -- (12.2,1);
	\draw (12.35,0) -- (12.35,1.4) node[left,at start] {$j+n-k$};
	\draw [decorate,decoration={coil,aspect=0}] (12.4,0.75) -- (14.35,0.75);
	\draw (14.4,0) -- (14.4,1.4) node[left,at start] {$n$};

	\draw[blue,dashed,thick] (-0.05,0.25) rectangle (11.2,1.15);
	\draw[violet,dashed,thick] (1.05,0.4) rectangle (12.3,1.35);	
	\end{tikzpicture}

\caption{Proof of \autoref{tetel}, case $j<k$}
\protect\label{thm11:j<k}
\end{figure}

By summing up for each $j$, we have

\begin{equation}
\phi(k,\sigma) = \summ_{j=1}^{k-1}\mu(j,\sigma)(\sigma-1)\sigma^{k-j-1} + \mu(k,\sigma)
\label{thm11_vege}
\end{equation}

This completes the proof.
\end{proof}

\begin{proof} (\autoref{tetel:2})

According to \autoref{lemma:5}, $\mu(j,\sigma)\leq \sigma^j -\sigma$ (if $j>1$).

In the proof of \autoref{tetel} at \eqref{thm11_vege} we saw for $k\geq 1$ and $n\geq 2k-1$ that
\begin{equation}
\phi(k,\sigma) = \mu(k,\sigma) + \summ_{j = 1}^{k-1}\mu(j,\sigma)(\sigma-1)\sigma^{k-j-1}.
\label{eq8}
\end{equation}
We can bound the right hand side of \eqref{eq8} from above as it follows:
\begin{equation}
\mu(k,\sigma) + \summ_{j = 1}^{k-1}\mu(j,\sigma)(\sigma-1)\sigma^{k-j-1} = \mu(k,\sigma) + \mu(1,\sigma)(\sigma-1)\sigma^{k-2} + \summ_{j=2}^{k-1}\mu(j,\sigma)(\sigma-1)\sigma^{k-j-1},
\end{equation}
which is by \autoref{lemma:5} at most 

\begin{equation}
\sigma^k-\sigma + \sigma(\sigma-1)\sigma^{k-2} + \summ_{j=2}^{k-1}(\sigma^j-\sigma)(\sigma-1)\sigma^{k-j-1} \leq \sigma^k + \sigma^k + \summ_{j=2}^{k-1}\sigma^j\sigma\sigma^{k-j-1} \leq k\sigma^k. 
\end{equation}

Thus, $\phi(k,\sigma)\leq k\sigma^k$, which means 
\begin{equation}
\sum_{k=1}^m \phi(k,\sigma) \leq \summ_{k=1}^m k\sigma^k \leq (m+1)\sigma^{m+1}.
\label{eq12}
\end{equation}

The left hand side of \ref{eq12} is an upper bound for the strings of growth at most $m$.

Let $m = \left\lfloor\frac{n}{2}\right\rfloor$.

As $\sigma^n \gg \frac{n}{2} \sigma^{\frac{n}{2}}$, this implies that in most cases the suffix tree of $S$ has at least $\frac{n}{2}$ more nodes than the suffix tree of $S[1,n-1]$.

Thus, a lower bound on the expectation of the growth of $S$ is 

\begin{equation}
\ee\(\gamma(S)\)\geq \frac{1}{\sigma^n}\(\frac{n}{2}\sigma^{\frac{n}{2}} + \(\sigma^n - \frac{n}{2}\sigma^{\frac{n}{2}}\)\(\frac{n}{2} + 1\) \),
\end{equation}
which is
\begin{equation}
	\frac{1}{\sigma^n}\(\frac{n+2}{2} \sigma^n + \(\frac{n}{2}-\frac{n(n+2)}{4}\)\sigma^{\frac{n}{2}}\) = cn,
\label{equat}
\end{equation}
with some $c$, if $n$ is large enough.

\end{proof}
With this, we have finished the proof and gave a quadratic lower bound on the average size of suffix trees.
	
\bibliographystyle{plain}

\end{document}